\documentclass[11pt,a4]{article}

\usepackage{enumerate}
\usepackage{latexsym} %Des symboles de maths en plus.
\usepackage{theorem}
\usepackage{graphics}
\usepackage{graphicx}
\usepackage{amsmath}
\usepackage{amssymb}%requis pour les polices mathbb

\newtheorem{defeng}{Definition}[section]
\newtheorem{theorem}[defeng]{Theorem}

\newtheorem{lemma}[defeng]{Lemma}

\newtheorem{corollary}[defeng]{Corollary}
{\theorembodyfont{\rmfamily} }
{\theorembodyfont{\rmfamily} }
{\theorembodyfont{\rmfamily} }
{\theoremstyle{break}\theorembodyfont{\rmfamily} }
{\theoremstyle{break}\theorembodyfont{\rmfamily} }

\newcounter{claim}

\newenvironment{proof}[1][]%
 {\noindent {\setcounter{claim}{0}\sc proof ---
   }{#1}{}}{\hfill$\Box$\vspace{2ex}} 

\newenvironment{claim}[1][]%
{\refstepcounter{claim}\vspace{1ex}\noindent{(\it\arabic{claim}){#1}{}}\it}{\vspace{1ex}}

\newenvironment{proofclaim}[1][]%
	{\noindent {}{#1}{}}{ This proves~(\arabic{claim}).\vspace{1ex}}

\newcommand{\sm}{\setminus} %prive de 

\usepackage{ifpdf}

\ifpdf
\fi

\title{On wheel-free graphs}

\author{Pierre Aboulker\thanks{Universit\'e Paris 7 -- Paris
    Diderot, LIAFA (France), \newline email: pierre.aboulker@liafa.jussieu.fr}\ ,
  Fr\'ed\'eric Havet\thanks{Projet Mascotte, I3S (CNRS, UNSA) and INRIA, Sophia Antipolis. Partly supported by the French
    \emph{Agence Nationale de la Recherche} under reference ANR-09-BLAN-0159;
 email: Frederic.Havet@inria.fr}\ , Nicolas
  Trotignon\thanks{CNRS, LIP -- ENS Lyon
    (France),\ email: nicolas.trotignon@ens-lyon.fr. \newline The first
    and last authors are partially supported by the French
    \emph{Agence Nationale de la Recherche} under reference
    \textsc{anr 10 jcjc 0204 01} and by PHC Pavle Savi\'c grant,
    jointly awarded by EGIDE, an agency of the French Minist\`ere des
    Affaires \'etrang\`eres et europ\'eennes, and Serbian Ministry of
    Education and Science.}}

\date{June 17, 2011}
\begin{document}

\maketitle

\begin{abstract}
  A wheel is a graph formed by a chordless cycle and a vertex that has
  at least three neighbors in the cycle.  We prove that every
  3-connected graph that does not contain a wheel as a subgraph is in
  fact minimally 3-connected.  We give a new proof of a theorem of
  Thomassen and Toft:  every graph that does not
  contain a wheel as a subgraph is 3-colorable.  
\end{abstract}

\vspace{2ex}

{\bf\noindent Key words: } Truemper configuration, wheel, cycle
through three vertices, coloring, minimally 3-connected graph.

\vspace{2ex}

{\bf\noindent AMS Classification: } 05C75, 05C38, 05C15

\section{Introduction}

A \emph{wheel} is a graph formed by a chordless cycle $C$, called the
\emph{rim}, and a vertex $v$ (not in $V(C)$), called the \emph{center}, such that the
center has at least three neighbors on the rim.  So, the complete
graph on four vertices is the smallest wheel.  When convenient, we
denote a wheel by $(C, v)$.  Wheels are one of the four
\emph{Truemper's configurations}, see~\cite{truemper}, that play a
role in several theorems on the structure of graphs and matroids.  Let
us see more precisely how wheels play some role in the description of
the structure of several graph classes.

A \emph{hole} in a graph is a chordless cycle on at least four
vertices.  The structure of a \emph{Berge graph} $G$, that is a graph
such that $G$ and its complement do not contain odd holes, is studied
in \cite{chudnovsky.r.s.t:spgt}.  The results obtained there famously
settled the Strong Perfect Graph Conjecture.  The proof goes through
several cases, and the last fifty pages of the proof deal with the case
when $G$ contains certain kinds of wheels.  Consequently
the structure of a Berge graph $G$ is simpler when $G$ does not
contain these kinds of wheels.  In addition, the structure of a graph $G$ with
no even holes is complex. A first decomposition theorem was given in \cite{conforti.c.k.v:eh1} and a better one in \cite{dsv:ehf}.  In the later paper,
very long arguments are devoted to situations when $G$ contains certain
kinds of wheels.  This suggests that graphs that do not contain a
wheel as an induced subgraph should have interesting structural
properties.  Understanding this structure might shed a new light on
the works listed above.  Since ``understanding the structure'' is a
slightly fuzzy goal, we adress the following precise open questions.

\vspace{2ex}

\noindent{\bf Question.} Is there a constant $c$ such that every graph
with no wheel as an induced subgraph is $c$-colorable?

\vspace{2ex}

\noindent{\bf Question.}  Is there a polynomial-time algorithm to
decide whether an input graph contains a wheel as an induced subgraph?

\vspace{2ex}

As observed by Esperet and Stehl\'ik~\cite{perso}, a classical
construction of triangle-free graphs with arbitrarily large chromatic
number, due to Zykov~\cite{Zyk49}, shows that the constant $c$ in the
first question must be at least~4.  This can also be deduced from the
graph represented in Figure~\ref{fig:R}: $R(3, 5)$ does not contain a
wheel as an induced subgraph, but has no 3-coloring because it has 13
vertices and stability number $4$.  
 \begin{figure}[hbtp]
\begin{center}
\includegraphics{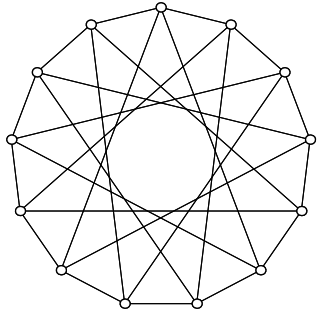}
\end{center}
\caption{The Ramsey graph $R(3, 5)$, that is the unique graph $G$
  satisfying $|V(G)|\geq 13$, $\alpha(G) = 4$ and $\omega(G) = 2$.\label{fig:R}}
\end{figure}
The above questions seem
difficult and it might be of interest to study subclasses of graphs
with no wheels as induced subgraphs.  Two such classes have already
been studied (but not motivated by the study of wheel-free graphs).
First, note that any wheel different from $K_4$ contains a cycle with
a unique chord.  So, graphs with no wheels as induced subgraphs is a
superclass of graphs with no cycles with a unique chord and no $K_4$.
These graphs have a precise structural description,
see~\cite{nicolas.kristina:one}.

A natural subclass of graphs with no wheel as an induced subgraph is
the class of graphs that do not contain a subdivision of a wheel as an
induced subgraph.  It is easy to see that this class of graphs is
precisely the class of graphs with no wheel and no subdivision of
$K_4$ as a induced subgraphs.  It turns out that the structure of a
graph $G$ with no induced subdivision of $K_4$ is investigated
in~\cite{nicolas:isk4}.  The proof goes through several cases: when
$G$ contains $K_{3, 3}$, when $G$ contains a \emph{prism} (another
Truemper's configuration, not worth defining here), and when $G$
contains a wheel.  This last case seems to be the most difficult: no
satisfactory structural description is found in this case, whereas
when excluding induced wheels, a very precise structure theorem is
given, with several consequences.  For instance, it is proved that
every graph that does not contain a subdivision of $K_4$ or a wheel
(as induced subgraphs) is 3-colorable.

\vspace{3ex}

Here we restrict our attention to another subclass: graphs with no
wheels \emph{as subgraphs}.  So, from here on, \emph{contain} and
\emph{-free} refer to the subgraph containement relation.

In Section~\ref{sec:zoo} we show several examples of wheel-free
graphs.  This will give insight to the reader and also shows that the
results presented here apply to a class of graphs that is not empty.
In Section~\ref{sec:Menger}, we prove several technical lemmas that
are all slight variations on Menger's Theorem needed in the rest of
the paper.  In Section~\ref{s:sd}, we study the connectivity of
wheel-free graphs.  The main result here is that every 3-connected
wheel-free graph is in fact minimally 3-connected.  As a direct
application, we prove that any wheel-free graph has a vertex of degree at
most~3.  This is a particular case of the following theorem due to
Turner.
\begin{theorem}[Turner \cite{turner:05}]
  \label{th:Turner}
  Let $k\geq 3$ be an integer and $G$ be a graph that does not contain
  a cycle together with a vertex that has at least $k$ neighbors in
  the cycle.  Then, $G$ has at least one vertex of degree at most
  $k$. 
\end{theorem}
Note that the result stated in \cite{turner:05} is slightly weaker
than Theorem~\ref{th:Turner}, but the proof given by Turner in
\cite{turner:05} exactly proves the version given here.  We still
include our proof that wheel-free graphs have vertices of degree at
most~3 in Section~\ref{s:sd}.  It is not as direct as Turner's, but it
illustrates some technics that we use later.  Also, with the same
method, we prove that every planar wheel-free graph has a vertex of
degree at most~2.

In a wheel-free graph, any vertex $v$ with three neighbors $x, y, z$
is such that deleting $v$ results in a graph where no cycle goes
through $x,y, z$.  In Section~\ref{s:3ic}, we recall a theorem due to
Watkins and Mesner~\cite{watkinsMesner:cycle}, that describes the
structure of a graph where no cycle goes through three given vertices
$x, y, z$.  We give a new shorter proof of this theorem.  In
Section~\ref{sec:Twin} we give an application of it: we prove that any
3-connected wheel-free graph contains a pair of vertices that are not
adjacent and have exactly the same neighborhood.  In fact, we need to
prove slightly more: the outcome holds not only for wheel-free graphs,
but also for a slightly larger class of graphs: \emph{almost
  wheel-free graphs} (to be defined later).  This result is then used
in Section~\ref{sec:col} to give a new proof of the following.

\begin{theorem}[Thomassen and Toft~\cite{thomassenToft:k4}]
  \label{th:2orT}
  Every wheel-free graph contains either a vertex of degree at most 2
  or a pair of non-adjacent vertices of degree~3 that have the same neighborhood.
\end{theorem}

As already observed by Thomassen and Toft, this implies the following.
\begin{corollary}
  \label{th:col}
  Every wheel-free graph is 3-colorable.  
\end{corollary}

\begin{proof}
  By induction on the number of vertices of a wheel-free graph $G$.
  If $|V(G)| = 1$, then $G$ is 3-colorable.  Otherwise, by
  Theorem~\ref{th:2orT}, either $G$ contains a vertex $w$ of degree at
  most~2, or a pair $\{u, v\}$ of non-adjacent vertices with the same
  neighborhood.  In the first case, we color $G - w$ by the induction
  hypothesis, and give to $w$ one of the three colors not used in its
  neighborhood.  In the second case, we color $G - u$ by the induction
  hypothesis, and give to $u$ the same color as $v$.
\end{proof}

In several papers about Truemper's configurations, rims of wheels are
required to be of length at least~4, i.e.\ $K_4$ does not count as a
wheel.  In Section~\ref{sec:long}, we show that this requirement does
not matter much for what we are doing here: any graph that does not
contain a wheel with a rim of length at least~4 is 4-colorable.

\section{A wheel-free zoo}
\label{sec:zoo}
%%%%%%%%%%%%%%%%%%%%%%%%%%

Wheel-free graphs with quite arbitrary shapes can be obtained by
taking any graph, and subdividing edges until every vertex of degree
at least~3 has all its neighbors (except possibly~2) of degree at
most~2.  Indeed in a graph obtained that way, no vertex can be the
center of a wheel.  But those graphs are not 3-connected (they have
vertices of degree~2).  In Figure~\ref{fig:some3c}, several
3-connected wheel-free graphs are represented.  They all have similar
shapes, but from that similarity, we could not deduce any general
construction for all 3-connected wheel-free graphs.
 \begin{figure}[hbtp]
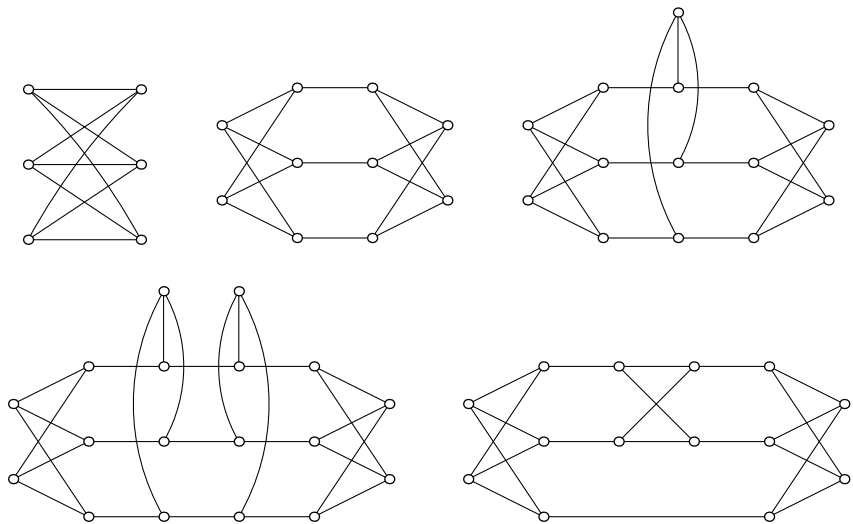

\begin{center}
\includegraphics{figWheels.2}\hspace{2em}
\includegraphics{figWheels.3}\hspace{2em}
\includegraphics{figWheels.4}\vspace{3ex}
\includegraphics{figWheels.5}\hspace{2em}
\includegraphics{figWheels.6}
\end{center}
\caption{Some 3-connected wheel-free graphs.\label{fig:some3c}}
\end{figure}

Note that all graphs from Figure~\ref{fig:some3c} are bipartite.
However, the graph represented in Figure~\ref{fig:cycle} on the left
has a cycle on 15 vertices, while being 3-connected and wheel-free.  On
the right is represented another wheel-free graph, with a seemingly
different shape.
\begin{figure}[hbtpp]
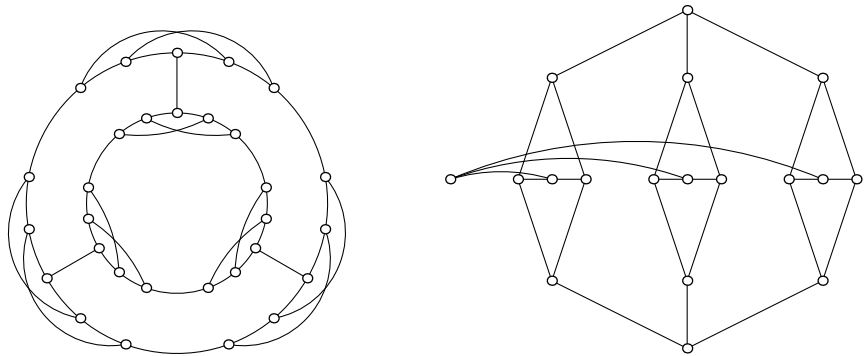

\begin{center}
\includegraphics{figWheels.7}\hspace{3em}
\includegraphics{figWheels.8}
\end{center}
\caption{A 3-connected wheel-free graph with a cycle on 15
  vertices (on the left). On the right, another wheel-free
  3-connected graph.\label{fig:cycle}}
\end{figure}

From all the graphs represented so far and in view of
Theorem~\ref{th:2orT}, it might be asked whether wheel-free graphs
that are subdivisions of 3-connected graphs are the graphs obtained
from a diamond (i.e. the graph obtained from the complete graph on four vertices by removing an edge) by randomly duplicating vertices of degree~3 with
neighbors of degree~2, and subdividing edges.  The graph represented
on the left in Figure~\ref{fig:2con} is a counter-example: it is a
wheel-free subdivision of a 3-connected graph, but it cannot be
obtained that way.  The graph represented on the right is 2-connected
in quite a strong sense: none of its subgraphs is 3-connected.
However, its minimum degree is~3.
\begin{figure}[hbtp]
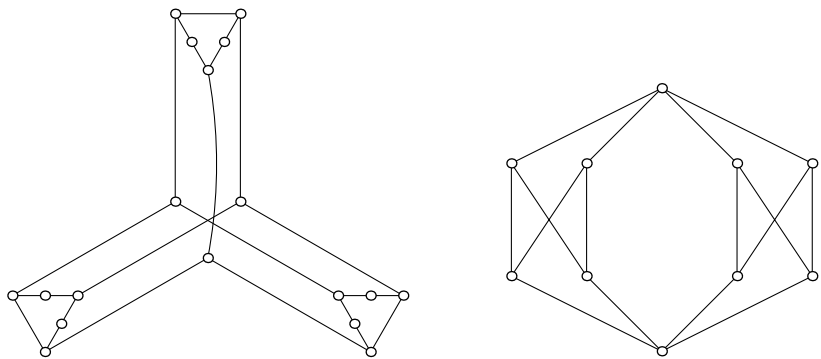

\begin{center}
\includegraphics{figWheels.9}\hspace{3em}
\includegraphics{figWheels.10}
\end{center}
\caption{Wheel-free graphs with connectivity 2\label{fig:2con}}
\end{figure}

\section{Variations on Menger's Theorem}
\label{sec:Menger}
%%%%%%%%%%%%%%%%%%%%%%%%%%%%%%

In this section, we present several lemmas that we will need later. They all
follow very easily from Menger's Theorem.  We refer
to~\cite{bondy.murty:book} for the statement of this theorem.

Paths of length 0 are allowed (they are made of one vertex).  We use
the following standard notation: when $G$ is a graph and $X$ a subset
of its vertices, we denote by $G - X$ the graph obtained from $G$ by
deleting vertices from $X$.  When $G - X$ is disconnected, we say that
$X$ is a \emph{cutset}.  When $v$ is a vertex, we sometimes write $G -
v$ instead of $G - \{v\}$.  A \emph{cutvertex} of a graph $G$ is a
vertex such that $G - v$ is disconnected.  When $e$ is an edge of~$G$,
we denote by $G\sm e$ the graph obtained from $G$ by deleting $e$
(note that the ends of $e$ are vertices of $G\sm e$).

In some
situation we need kinds of separations where some sets are allowed to
intersect, so we need to define them precisely.  Let $\{a, b\}$, $\{c,
d\}$ be two sets of vertices of a graph $G$ such that $a\neq b$ and $c\neq
d$.  Note that the two sets may intersect or be equal.  A vertex $v$
is an \emph{$(\{a, b\}, \{c, d\})$-separator} if in $G - v$, there is
no path with one end in $\{a, b\}$ and the other end in $\{c, d\}$.
Note that $v$ can be one of $a, b, c, d$.  The following is a
rephrasing of Menger's Theorem in a particular case.

\begin{lemma}
  \label{l:menger2}
  Let $G$ be a graph and $\{a, b\}$, $\{c, d\}$ two sets of vertices of
  $G$ such that $a\neq b$ and $c\neq d$.  Either there exists an
  $(\{a, b\}, \{c, d\})$-separator or there exists two vertex-disjoint
  paths from $\{a, b\}$ to $\{c, d\}$.
\end{lemma}

Note that the statement above is true when $\{a, b\} = \{c,d\}$.  In
this case, the two paths are of length 0.  

Let $k\geq 1$ be an integer, $G$ a graph, $Y\subseteq V(G)$ a set on
at least $k$ vertices, and $x\notin Y$ a vertex of $G$.  A family of
$k$ paths from $x$ to $Y$ whose only common vertex is $x$ and whose
internal vertices are not in $Y$, is called a \emph{$k$-fan from $x$
  to $Y$}.  The next two results are classical
(see~\cite{bondy.murty:book}).

\begin{lemma}[Fan Lemma] 
  \label{l:fanLemma}
  Let $G$ be a $k$-connected graph, $x$ a vertex of $G$ and $Y$ a
  subset of $V(G) \sm \{x\}$ of cardinality at least $k$. Then
  there is a $k$-fan from $x$ to $Y$.
\end{lemma}

\begin{theorem}[Dirac \cite{dirac:cycle}, see also~\cite{bondy.murty:book}]
  \label{th:cycleThroughK}
  Let $S$ be a set of $k$ vertices of a $k$-connected graph $G$ where
  $k \ge 2$.  Then there is a cycle of $G$ that contains all the
  vertices of $S$.
\end{theorem}

Let $C$ be a cycle of a graph $G$ and $v$ a vertex not in $C$.  We say
that a set on at most two vertices $\{a, b\}$ is a
\emph{$(v,C)$-separator} if $v \notin \{a, b\}$ and if $G - \{a, b\}$
contains no path from $v$ to $V(C)\sm \{a, b\}$.  Note that $\{a, b\}$
and $C$ may intersect.  The following is another rephrasing of
Menger's Theorem in a particular case.

\begin{lemma}
  \label{l:fan}
  In a graph $G$, if $C$ is a cycle and $v$ is a vertex not in $C$
  such that there exists no $(v, C)$-separator, then $G$ contains a
  3-fan from $v$ to $C$.
\end{lemma}

\begin{lemma}
  \label{l:2cutsetCycle}
  Let $G$ be a graph, $C$ a cycle of $G$ and $x$ a vertex not in $C$
  such that there exists no ($x$, $C$)-separator.  Let $y, z$ be two
  distinct vertices of $C$.  Then there exists a cycle of $G$ that
  goes through $x$, $y$ and $z$.
\end{lemma}

\begin{proof}
  Cycle $C$ is edge-wise partitioned into two paths $Q =y\dots z$ and
  $R = y \dots z$.  By Lemma~\ref{l:fan}, there exists a 3-fan made of
  $P_1 = x \dots c_1$, $P_2 = x \dots c_2$ and $P_3 = x \dots c_3$,
  from $x$ to $C$.  From the pigeon-hole principle, at least two
  vertices of $\{c_1, c_2, c_3\}$ are in $Q$ or in $R$, say $c_1, c_2
  \in Q$.  Suppose up to a relabelling that $y, c_1, c_2, z$ appear in
  this order along $Q$.  Then $yQc_1P_1xP_2c_2QzRy$ is a cycle that
  goes through $x, y, z$.
\end{proof}

The following is the basic tool to characterize the situation when no
cycle goes through three given vertices of a 2-connected graph.  Note
that contrary to Theorem~\ref{th:3inaCycle}, it is not an ``if and only statement''.

\begin{lemma}
  \label{l:cycleOrTheta}
  Let $G$ be a 2-connected graph and $x, y, z$ be three vertices of~$G$.  Then either
  \begin{itemize}
  \item a cycle of $G$ goes through $x, y, z$; or
  \item $x, y, z$ are distinct and there exist two distinct vertices
    $t_A, t_B \notin \{x, y, z\}$ and six internally vertex-disjoint
    paths $P_A = t_A \dots x$, $P_B = t_B \dots x$, $Q_A = t_A \dots
    y$, $Q_B = t_B \dots y$, $R_A = t_A \dots z$ and $R_B = t_B \dots
    z$.
 \end{itemize}
\end{lemma}

\begin{proof}
  Since $G$ is 2-connected, we know that $x$, $y$ and $z$ are distinct
  (or a cycle goes through them) and there exists a cycle $C$ that
  goes through $x, z$.  Cycle $C$ is edge-wise partitioned into two
  paths $S_A$ and $S_B$ from $x$ to~$z$.  Since $G$ is 2-connected, if
  $y\notin V(C)$, then there exists a 2-fan from $y$ to $C$, formed by
  $Q_A = y \dots t_A$ and $Q_B = y \dots t_B$ say.  If $t_A, t_B \in
  V(S_A)$, then up to symmetry, $x, t_A, t_B, y$ appear in this order
  along $S_A$ and $x S_A t_A Q_A y Q_B t_B S_A z S_B x$ is a cycle
  through $x, y, z$.  Similarly, if $t_A, t_B \in V(S_B)$, then one
  finds such a cycle.  Hence, we may assume $t_A \in V(S_A) \sm \{x,
  z\}$ and $t_B \in V(S_B)\sm \{ x, z\}$.  We let $P_A = x S_A t_A$,
  $R_A = z S_A t_A$, $P_B = x S_B t_B$ and $R_B = z S_B t_B$.
\end{proof}

\section{Connectivity of wheel-free graphs}
\label{s:sd}
%%%%%%%%%%%%%%%%%%%%%%%%%%%%%%%%%%%%

The connectivity of a graph $G$ is denoted by $\kappa(G)$.  An edge
$e$ of a graph $G$ is \emph{essential} if $\kappa(G \sm e) <
\kappa(G)$.  A graph with connectivity $k$ and such that all its edges
are essential is \emph{minimally $k$-connected}.  Our goal in this
section is to prove that every 3-connected wheel-free graph is
minimally 3-connected.  This will be of use because of the following
well known theorems.

\begin{theorem}[Mader \cite{mader:mkc}, see also \cite{bollobas:egt}] 
  \label{maderCycle}
  If $G$ is a minimally 3-connected graph, then every cycle of $G$
  contains a vertex of degree~3.
\end{theorem}

\begin{theorem}[Mader \cite{mader:mkc}, see also \cite{bollobas:egt}] 
  \label{mader}
  If $G$ is a minimally 3-connected graph, then $G$ has at least
  $\frac{2|V(G)|+2}{5}$ vertices of degree~3.
\end{theorem}

 For every graph $G$, we denote by $W(G)$ the set of
all vertices $u$ of $G$ such that at least one wheel of $G$ is
centered at~$u$.

\begin{lemma}
  \label{l:no4}
  If $G$ is a $4$-connected graph, then $W(G)=V(G)$.  In particular a
  wheel-free graph has connectivity at most~3.
  \end{lemma}

\begin{proof}
  If $G$ is 4-connected, then any vertex $v$ has at least four
  neighbors.  Since $G - v$ is 3-connected, by
  Theorem~\ref{th:cycleThroughK}, it contains a cycle going through
  three neighbors of $v$.  Together with $v$, this cycle forms a wheel
  centered at $v$.
\end{proof}

If $A \subset V(G)$, we denote by $N(A)$ the set of vertices from
$V(G) \setminus A$ adjacent to at least one vertex of $A$.  When
$F\subseteq V(G)$, we denote by $\overline{F}$ the set $V(G) \setminus
(F \cup N(F))$.  We say that $F$ is a \emph{fragment} of $G$ if
$|N(F)|=\kappa(G)$ and $\overline{F} \ne \emptyset$ (note that if $F$
is a fragment of $G$, then so is $\overline{F}$).  An \textit{end} of
$G$ is a fragment not containing other fragments as proper subsets. It
is clear that any fragment $F$ contain an end, and that consequently
all graphs contain at least two disjoints ends: one in $F$, another
one in $\overline{F}$.

\begin{lemma}
  \label{l:2conBlock}
  Let $G$ be a wheel-free graph such that $\kappa(G) = 2$ and $F$ be
  an end of $G$ such that $|F|\geq 2$ and $N(F) = \{a, b\}$.  Let
  $G_F$ be the graph obtained from $G[F \cup \{a, b\}]$ by adding the
  edge $ab$ (if it is not there already).  Then $G_F$ is 3-connected
  and $W(G_F) \subseteq \{a, b\}$.
\end{lemma}

\begin{proof}
  Note that $|V(G_F)| \geq 4$.  Let us suppose by way of contradiction
  that $G_F$ admits a cutset of cardinality~2, say $\{u, v\}$.  The
  set $\{a, b\}$ is clearly not a cutset of $G_F$, so $|\{u,v\} \cap
  \{a,b\}| < 2$.  If $|\{u,v\} \cap \{a,b\}| =1$, then $\{u,v\}$ is
  also a cutset of $G$ which has a fragment strictly included in $F$,
  a contradiction.  In the same way, if $\{u,v\} \cap \{a, b\} =
  \emptyset$, then, since $ab \in E(G_F)$, $a$ and $b$ are in the same
  component of $G_F\setminus \{u,v\}$.  Hence any component of
  $G_F-\{u,v\}$ not containing $a$ and $b$, is also a component of $G
  -\{u,v\}$ and thus a fragment strictly included in $F$, a
  contradiction.  So, $G_F$ does not contain a cutset of cardinality~2
  and, as $|V(G_F)| \ge 4$, $G_F$ is 3-connected.
  
  Suppose that $G_F$ contains a wheel $(C, w)$.  Since $G$ is
  wheel-free, the edge $ab$ must be an edge of that wheel, and
  $ab\notin E(G)$.  If $ab$ is an edge of $C$, then a wheel of $G$ is
  obtained by replacing $ab$ with a path from $a$ to $b$ with internal
  vertices in $\overline{F}$, a contradiction.  Hence, $ab$ is an edge
  incident to the center of $(C, w)$, so $w\in \{a, b\}$.  This proves
  $W(G_F) \subseteq \{a, b\}$.
\end{proof}

\begin{lemma} 
  \label{essential} 
  If a 3-connected graph $G$  contains an edge $e =ab$ that is
  not essential, then $\{a, b\} \subseteq W(G)$.
\end{lemma}

\begin{proof}
  Since $G \setminus ab$ is $3$-connected, there exist three
  vertex-disjoint paths $T_1=aa_1 \dots b$, $T_2 = aa_2 \dots b$ and
  $T_3=aa_3 \dots b$ in $G \setminus ab$.

  In $G-a$, which is $2$-connected, we may assume that no cycle goes
  through $a_1$, $a_2$ and $a_3$ (otherwise $a\in W(G)$).  So, by
  Lemma~\ref{l:cycleOrTheta} applied to $G-a$, there exist two
  vertices $u, v$ and six internally vertex-disjoint paths $P_1 = a_1
  \dots u$, $P_2 = a_2 \dots u$, $P_3 = a_3 \dots u$, $Q_1 = a_1 \dots
  v$, $Q_2 = a_2 \dots v$ and $Q_3 = a_3 \dots v$.  We set $X = P_1
  \cup P_2 \cup P_3 \cup Q_1 \cup Q_2 \cup Q_3$.

  Because of $T_1, T_2, T_3$, either $b\in X$, in which case we
  suppose $b\in P_1$, or there exists a 3-fan from $b$ to $X$ in
  $G-a$.  When $b\notin X$, from the pigeon-hole principle, at least
  two paths from this 3-fan end in $P_1\cup P_2 \cup P_3$ or in
  $Q_1\cup Q_2 \cup Q_3$.  So, up to symmetry, if $b\notin X$, then we
  may assume that there exists a 2-fan from $b$ to $P_1\cup P_2$.  It
  follows that (wherever $b$) there is a cycle in $G - a$ that goes
  through $a_1$, $a_2$ and $b$.  Together with $a$, this cycle forms a
  wheel centered at $a$.  This proves $a\in W(G)$, and $b\in W(G)$ can
  be proved similarly.
\end{proof}

A graph is \emph{almost wheel-free} if $W(G)$ is either empty, or made
of a single vertex of degree~3, or made of two adjacent vertices, both
of degree~3 (this notion will be used more in the next sections).  By definition, every
wheel-free graph is almost wheel-free.

\begin{corollary}
  \label{col:3con}
  If $G$ is a 3-connected almost wheel-free graph, then $G$ is
  minimally 3-connected.
\end{corollary}

\begin{proof}
  Since $G$ is 3-connected, by Lemma~\ref{l:no4}, $G$ has
  connectivity~3.  Let $e=uv$ be an edge of $G$.  Suppose for a
  contradiction that $e$ is not essential. Then $\deg(u), \deg(v) \geq
  4$, and, by Lemma~\ref{essential}, $u, v \in W(G)$.  This
  contradicts the fact that $G$ is almost wheel-free.  Hence, all
  edges of $G$ are essential and so $G$ is minimally 3-connected.
\end{proof}

It is tempting to use Corollary~\ref{col:3con} to give a direct proof
of the next theorem.  Indeed, consider the following class $C$ of
graphs: graphs such that any subgraph has connectivity at most 2 or is
minimally 3-connected.  By Lemma~\ref{l:no4} and
Corollary~\ref{col:3con}, any wheel-free graph is in $C$.  Since $C$
is made of minimally 3-connected graphs, which have vertices of
degree~3 by Theorem~\ref{mader}, and of graphs that are even less
connected, it could be that any graph in $C$ has a vertex of degree at
most~3.  But unfortunately, there exist graphs in $C$ of minimum
degree~4 (they contain wheels), see Figure~\ref{fig:mindeg4}.  Note
also that the next theorem is best possible in some sense, since many
wheel-free graphs have no vertex of degree less than 3, as shown by
the graphs represented on Figures~\ref{fig:some3c}
and~\ref{fig:cycle}.
 \begin{figure}[hbtp]
\begin{center}
\includegraphics{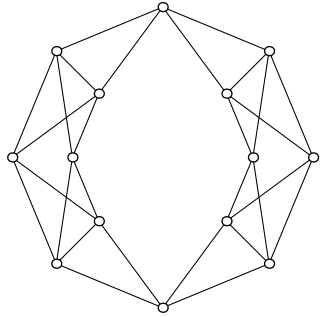}
\end{center}
\caption{A graph with minimum degree 4 and no 3-connected subgraph.\label{fig:mindeg4}}
\end{figure}

\begin{theorem}
  \label{th:deg3}
  If $G$ is a wheel-free graph on at least two vertices, then
  $G$ has at least two vertices of degree at most~3.
\end{theorem}

\begin{proof}
  Our proof is by induction on $|V(G)|$, the result holding trivially when
  $|V(G)| \leq 4$.  
  
  If $G$ is not connected, then by the induction hypothesis,
  each of its components has at least one vertex of degree at most~3,
  so $G$ contains at least two vertices of degree at most~3.
  
    If $G$ has a cutvertex $a$, then let $C_1$ and $C_2$ be components of
  $G - a$.  By the induction hypothesis, $G[C_1\cup \{a\}]$ and
  $G[C_2\cup \{a\}]$ have each two vertices of degree at most~3. Thus at least
  one of them is distinct from $a$ and thus is also a vertex of degree
  at most~3 in $G$.  Hence, $C_1$ and $C_2$ have each at least one
  vertex of degree at most~3 in $G$.

  If $G$ is 3-connected, then, by Corollary~\ref{col:3con}, it is
  minimally 3-connected, so by Theorem~\ref{mader}, it has at least
  two vertices of degree at most~3.

  Assume finally that $G$ has connectivity~2.  Let $F$ and $F'$ be two
  disjoints ends of $G$.  It is enough to prove that each of $F, F'$
  contains at least one vertex of $G$ of degree~3.  Let us prove it
  for $F$, the proof being similar for~$F'$.
  
    If $|F| \leq 2$, this
  is easy to check.  So, suppose $|F|\geq 3$.
Let $\{u,v\} = N(F)$ and $G_F$ be the graph as in
  Lemma~\ref{l:2conBlock}.  Hence $G_F$ is 3-connected. Moreover every edge
  $e\neq uv$ of $H$ is essential. Indeed if an edge different from $uv$
  were not essential,  then by Lemma~\ref{essential} some vertex $a
  \notin \{u, v\}$ would be the center of some wheel of $G_F$, a
  contradiction to Lemma~\ref{l:2conBlock}.

  Note that $G_F\sm uv$ is a subgraph of $G$ and so is wheel-free. Assume first that $G_F \sm uv$ is
  3-connected. Then
   by Corollary~\ref{col:3con}, it is minimally 3-connected.  So, by
  Theorem~\ref{mader}, $G_F \sm uv$ contains at least three vertices
  of degree at most~3. One of those is distinct from $u$ and $v$ and thus has degree at most~3
  in~$G$.  Assume finally that $G_F\sm uv$ is not 3-connected.  This
  means that $uv$ is essential in $G_F$, so, all edges of $G_F$ are
  essential, so $G_F$ is minimally 3-connected (note that $G_F$ may
  contain wheels).  We conclude as above by using Theorem~\ref{mader}
  in $G_F$.
\end{proof}

With slight modifications in the proof, we shall now prove that any
wheel-free planar graph on at least two vertices contains at least two
vertices of degree at most~2.  In fact, the key property that we use
is that a planar graph does not contain a subdivision of $K_{3, 3}$.

\begin{lemma}
  \label{l:hubK33}
  If $G$ is a 3-connected graph that contains no subdivision of $K_{3,
    3}$, then $W(G)=V(G)$.
\end{lemma}

\begin{proof}
  Let $v$ be a vertex of $G$.  It has at least three neighbors $x, y,
  z$.  If no cycle goes through them, then let $P_A, Q_A, R_A, P_B,
  Q_B, R_B$ be the six paths of $G - v$ (which is 2-connected) whose
  existence is proved in Lemma~\ref{l:cycleOrTheta}.  Together with
  $v$, they form a subdivision of $K_{3, 3}$, a contradiction.  Hence
  a cycle $C$ goes through $x,y,z$, so $(C,v)$ is a wheel centered at
  $v$.
\end{proof}

\begin{theorem}
  If $G$ is a wheel-free graph on at least two vertices that contains
  no subdivision of $K_{3, 3}$, then $G$ has at least two vertices of
  degree at most~2.
\end{theorem}

\begin{proof}
  The proof is very similar to the proof of Theorem~\ref{th:deg3}.  We
  start with a graph $G$ on at least two vertices.  As in the proof of
  Theorem~\ref{th:deg3}, we may assume that $G$ is 2-connected.  So,
  by Lemma~\ref{l:hubK33}, $G$ has connectivity~2.  We consider two
  disjoint ends $F$ and $F'$ of $G$.  It is enough to prove that both
  of them have cardinality 1.  So, suppose for a contradiction that
  $F$ has cardinality at least~2.  Let $\{u, v\} = N(F)$, and $G_F$ be
  the graph as in Lemma~\ref{l:2conBlock}.  So $G_F$ is 3-connected.
  In addition, it contains no subdivision of $K_{3, 3}$. Indeed if a subgraph $H$
  of $G_F$ is a subdivision of $K_{3, 3}$, then $H$ contains the edge
  $uv$.  So replacing $uv$ by some path from $u$ to $v$ with
  internal vertices in $\overline{F}$ yields a subdivision of $K_{3, 3}$
  in~$G$, a contradiction.  Hence, by Lemma~\ref{l:hubK33}, any vertex
  of $G_F$ is the center of a wheel.  In particular, $G_F$ contains a
  wheel whose center is not among $u, v$, a contradiction to
  Lemma~\ref{l:2conBlock}.
\end{proof}

\section{Three vertices in a cycle}
\label{s:3ic}

%%%%%%%%%%%%%%%%%
\begin{figure}
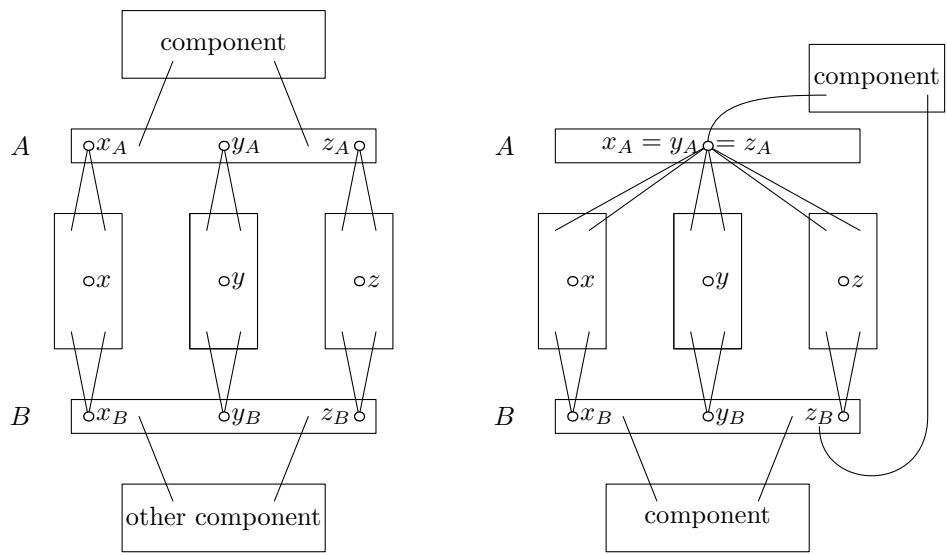
%[hbtp]
\begin{center}
 \includegraphics{fig3vertexCycle.1}\hspace{3em}
  \includegraphics{fig3vertexCycle.2}
\end{center}
\caption{Two graphs with a splitter with respect to $x$, $y$ and
  $z$.\label{fig:splitter}}
\end{figure}

The problem of deciding whether a cycle exists through three given
vertices of a graph is solved from an algorithmic point of view.
There is a linear time algorithm by LaPaugh and
Rivest~\cite{lapaughR:80}.  A simpler algorithm is given by Fleischner
and Woeginger~\cite{fleischnerW:92}.  They also give a certificate
when the answer is no, but it relies on the decomposition tree of a
graph into its triconnected components,
see~\cite{hopcroft.tarjan:3con}.  What we need here is a certificate
given in terms of cutsets.  The aim of this section is to state such a
certificate, whose existence is proved by Watkins and
Mesner~\cite{watkinsMesner:cycle} (see also \cite{gould:cycle} for a
survey about problems of cycles through prescribed elements of a
graph).  We state this result in a different way (for a more
convenient use in the next section), but the equivalence between the
two versions is immediate.  Also, we give a new proof, which is
slightly shorter and, we believe, simpler.

Let $G$ be a graph and $x, y, z$ three distinct vertices.  A pair $(A,
B)$ of two disjoint non-empty sets of vertices is a \emph{splitter
  with respect to x, y, z} if (see Figure~\ref{fig:splitter}):

\begin{enumerate}[(i)]
\item\label{i:xyz} $x, y, z$ are respectively in three distinct
  components $X, Y, Z$ of $G - (A \cup B)$.

\item\label{i:xayaza} All edges between $X$ and $A$ (resp.\ $Y$ and
  $A$, $Z$ and $A$) are incident to a unique vertex $x_A$ (resp.\
  $y_A$, $z_A$) of $A$.  

\item\label{i:xbybzb} All edges between $X$ and $B$ (resp.\ $Y$ and
  $B$, $Z$ and $B$) are incident to a unique vertex $x_B$ (resp.\
  $y_B$, $z_B$) of $B$.  

\item\label{i:AB} $A = \{x_A, y_A, z_A\}$, $B = \{x_B, y_B, z_B\}$.

\item\label{i:1or3} Either $|A| = 1$ or $|A| = 3$. Either $|B| = 1$ or
  $|B| = 3$.

\item\label{i:2con} $G - X$, $G - Y$ and $G - Z$ are 2-connected.

\item\label{i:component} If $|A| = 3$ and $|B| = 3$, then every edge
  between $A$ and $B$ is one of $x_Ax_B$, $y_Ay_B$ or $z_Az_B$, and
  every component $D$ of $G - (A \cup B)$ is such that $N(D)$ is
  included in either $A$, $B$, $\{x_A, x_B\}$, $\{y_A, y_B\}$ or
  $\{z_A, z_B\}$.
\end{enumerate}

\begin{theorem}[Watkins and Mesner~\cite{watkinsMesner:cycle}]
  \label{th:3inaCycle}
  Let $G$ be a 2-connected graph and $x, y, z$ three vertices of $G$.
  No cycle goes through $x, y, z$ if and only if $G$ admits a splitter
  with respect to $x, y, z$.
\end{theorem}

\begin{proof}
  If $G$ has a splitter it is a routine matter to check  that no
  cycle exists through $x, y, z$.

  Conversely, suppose that no cycle goes through $x, y, z$.  We apply
  Lemma~\ref{l:cycleOrTheta} to $G$ and $x, y, z$: this defines six
  paths $P_A$, $P_B$, $Q_A$, $Q_B$, $R_A$ and~$R_B$.  There must exist
  a pair $\{x_A, x_B\}$ that is an $(x$, $y Q_A t_A R_A z R_B t_B Q_B
  y)$-separator, for otherwise, by Lemma~\ref{l:2cutsetCycle}, there
  is a cycle through $x, y, z$.  Because of the paths $P_A$ and $P_B$,
  we must have, up to a relabelling, $x_A \in P_A - x$ and $x_B \in
  P_B - x$.  Let $X$ be the component of $x$ in $G - \{x_A, x_B\}$.
  We choose $x_A$ and $x_B$ so as to maximize the size of $X$.
  Similarly, there exists a $(y, xP_At_AR_AzR_Bt_BP_Bx)$-separator
  $\{y_A, y_B\}$, where $y_A \in Q_A - y$ and $y_B \in Q_B - y$.  We
  choose $y_A$ and $y_B$ so as to maximize the size of the component
  $Y$ of $y$ in $G - \{y_A, y_B\}$.  Finally, there exists a $(z, x
  P_A t_A Q_AyQ_Bt_BP_Bx)$-separator $\{z_A, z_B\}$ where $z_A \in R_A
  - z$ and $z_B \in R_B - z$.  We choose $z_A$ and $z_B$ so as to
  maximize the size of the component $Z$ of $z$ in $G - \{z_A, z_B\}$.
  Set $P = x_A P_A x P_B x_B$, $Q = y_A Q_A y Q_B y_B$ and $R = z_A
  R_A z R_B z_B$.

\vspace{6pt}

Set $A = \{x_A, y_A, z_A\}$ and $B = \{x_B, y_B, z_B\}$.  Our goal is
now to prove that $(A,B)$ is a splitter with respect to $x, y, z$.
Conditions~(\ref{i:xyz}) to~(\ref{i:AB}) are satisfied from the
definition of $x_A, \dots, z_B$.  

Let us prove~(\ref {i:1or3}).  Suppose $|A| = 2$.  Hence, up to
symmetry, we may assume that $x_A = y_A$ and $z_A \neq x_A$.  We see
that $x_B \neq y_B$ for otherwise, $\{x_A, x_B\}$ is a $(z, x P_A t_A
Q_AyQ_Bt_BP_Bx)$-separator that contradicts the maximality of $Z$.  If
in $G - (Z \cup \{x_A\})$ there exists a $(\{z_A, z_B\}, \{x_B,
y_B\})$-separator $u$, then $u\in V(R_B)$ and $\{x_A, u\}$ is a $(z, x
P_A t_A Q_AyQ_Bt_BP_Bx)$-separator in $G$ that contradicts the
maximality of $Z$.  Hence, by Lemma~\ref{l:menger2}, there exists two
vertex-disjoint paths in $G - (Z \cup \{x_A\})$ from $\{z_A, z_B\}$ to
$\{x_B, y_B\}$. Together with $x_A P x_B$, $x_A Q y_B$ and $z_A R
z_B$, they  form a cycle through $x, y, z$, a contradiction.  This proves
$|A| = 1$ or $|A| = 3$.  Similarly, we can prove that $|B| = 1$ or
$|B| = 3$.

Let us now prove~(\ref{i:2con}). Suppose for a contradiction that
$G-X$ has a cutvertex $u$.  If $u$ is not in $P_A - x_A$ or in $P_B - x_B$, then
$u$ is a cutvertex of $G$, a contradiction to the 2-connectivity of
$G$.  So, up to symmetry, $u \in P_A - x_A$.  Thus $\{x_B, u\}$ is an $(x$,
$y Q_A t_A R_A z R_B t_B Q_B y)$-separator, a contradiction to the
maximality of $X$.  Hence $G - X$ is 2-connected.  Similarly, $G - Y$
and $G - Z$ are 2-connected.  This proves~(\ref{i:2con}).

  \vspace{6pt}

  Let us now show an intermediate statement.

   \begin{claim}
     \label{c:nocut}
     If $|A| = 3$ and $|B| = 3$, then there exist two 
     subgraphs $G_A$ and $G_B$ of $G -  (X\cup Y \cup Z)$ such that:
     \begin{enumerate}[(a)]
     \item\label{i:disjoint} $G_A$ and $G_B$ are vertex-disjoint;
     \item\label{i:cont} $G_A$ contains $x_A, y_A, z_A$ and $G_B$
       contains $x_B, y_B, z_B$;
     \item\label{i:cut} $G_A$ and $G_B$ are 2-connected.
     \end{enumerate}
   \end{claim}

   \begin{proofclaim}   
     For any graph $H$ we define the parameter $c(H) = \sum_{v\in
       V(H)} (\mathrm{comp}(H-v) -1)$ where $\mathrm{comp}(H-v)$
     denotes the number of components of $H- v$.

     Let $(G_A,G_B)$ be a pair of connected graphs that
     satisfy~(\ref{i:disjoint}) and~(\ref{i:cont}) and such that
     $c(G_A) + c(G_B)$ is minimum.  We refer to this property as the
     \emph{minimality of $(G_A, G_B)$}.  Note that such a pair $(G_A,
     G_B)$ exists because the two graphs $(x_A P_At_A) \cup (y_A Q_A
     t_A) \cup (z_A R_A t_A)$ and $(x_B P_B t_B) \cup (y_B Q_B t_B)
     \cup (z_B R_B t_B)$ are connected and satisfy~(\ref{i:disjoint})
     and~(\ref{i:cont}).
     
     Let us prove~(\ref{i:cut}) by contradiction.  Therefore suppose
     that one of $G_A$ and $G_B$, say $G_A$, has a cutvertex $v_A$.
     If $x_A, y_A, z_A$ are all in the same graph $G_A[C \cup
     \{v_A\}]$ where $C$ is a component of $G_A - v_A$, then
     $(G_A[C\cup \{v_A\}], G_B)$ contradicts the minimality of $(G_A,
     G_B)$.  So, without loss of generality, we may assume that $x_A$
     is in a component $C_A$ of $G_A - v_A$ and that $y_A, z_A$ are
     not in $C_A$.  We suppose moreover that $v_A$ is chosen so as to
     maximize the size of $C_A$.  If $G_B$ admits a vertex $v_B$ such
     that $x_B$ is in a component $C_B$ of $G_B - v_B$ and $y_B, z_B$
     are not in $C_B$, then we choose $v_B$ such that the component
     $C_B$ of $G_B - v_B$ that contains $x_B$ is maximal.  Else, we
     set $v_B = x_B$ and $C_B = \emptyset$.

     In $G$, $\{v_A, v_B\}$ is not an $(x$, $y Q_A t_A R_A z R_B t_B
     Q_B y)$-separator because of the maximality of $X$.  So, there
     exists a path $S$ of $G$ with one end $s$ in $C_A\cup C_B$, the
     other end $s'$ in $(G_A \cup G_B) - (\{v_A, v_B\} \cup C_A \cup
     C_B)$, no internal vertex of which is in $V(G_A) \cup V(G_B) \cup
     X \cup Y \cup Z$ and no edge of which is in $E(G_A) \cup E(G_B)$.
     Up to symmetry, we assume $s \in C_A$.  We have $s' \in G_B -
     (\{v_B\} \cup C_B)$, for otherwise $(G_A \cup S, G_B)$
     contradicts the minimality of $(G_A, G_B)$ because $\mathrm{comp}(G_A-v_A)> \mathrm{comp}(G_A\cup S -v_A)$ and
     for every internal vertex $t$ of $S$, the graph $G_A\cup S -t$ is
     connected.  
     
     If $G_B$ admits an $(\{x_B, s'\}, \{y_B, z_B\})$-separator $w$,
     then $w$ is such that $x_B$ and $v_B$ are in a component $C$ of
     $G_B - w$ and $y_B, z_B$ are not in $C$.   So $w$ contradicts the
     maximality of $C_B$.  Thus in $G_B$ no $(\{x_B, s'\}, \{y_B,
     z_B\})$-separator exists. Hence, by Lemma~\ref{l:menger2}, in
     $G_B$, up to symmetry between $y_B$ and $z_B$, there are two
     vertex-disjoint paths $T_B = x_B \dots y_B$ and $T'_B = s' \dots
     z_B$.  In $C_A$, there is a path $T_A$ from $s$ to $x_A$.  In $G_A -
     C_A$ there is a path $T'_A$ from $y_A$ to $z_A$ (because $G_A -
     C_A$ is connected since $G_A$ is connected and $C_A$ is a
     component of $G_A - v_A$).  We observe that $P\cup Q \cup R \cup
     S \cup T_A \cup T'_A \cup T_B \cup T'_B$ is a cycle that goes
     through $x, y, z$, a contradiction.
   \end{proofclaim}

   To finish the proof, suppose for a contradiction that
   Conditions~(\ref{i:component}) fails.  This means whithout loss of
   generality that $|A| = |B| = 3$ and there exists a path $S$ from
   $x_A$ to $y_B$ in $G$ which contains no vertex of $\{y_A, z_A, x_B,
   z_B\}$.  Path $S$ has one end in $G_A$ and one end in $G_B$ and
   $G_A, G_B$ are vertex-disjoint, so $S$ contains a subpath $S'$ with
   one end $s_A$ in $G_A$, one end $s_B$ in $G_B$, no internal vertex
   in $G_A \cup G_B$ and no edge of $S'$ is an edge of $G_A \cup G_B$.
   Note that $S'$ contains no vertex of $\{y_A, z_A, x_B, z_B\}$.  We
   reach a contradiction by considering two cases.

   {\noindent\bf Case~1:} in $G_A$ there exist two vertex-disjoint
   paths $T_A = x_A \dots s_A$, $T'_A = y_A \dots z_A$; or in $G_B$,
   there exist two vertex-disjoint paths $T_B = y_B \dots s_B$, $T'_B
   = x_B \dots z_B$.  Up to symmetry, we suppose that $T_A$ and $T'_A$
   exist.  Let us apply Lemma~\ref{l:menger2} in $G_B$.  An $(\{x_B,
   s_B\}, \{y_B, z_B\})$-separator would be a cutvertex of $G_B$, a
   contradiction to~(\ref{c:nocut}).  So, there exist two
   vertex-disjoint paths $T_B, T'_B$ between $\{x_B, s_B\}$ and
   $\{y_B, z_B\}$.  Note that $\{x_B, s_B\}$ has two elements because
   $S'$ has no vertex in $\{y_A, z_A, x_B, z_B\}$ (but $s_B=y_B$ is
   possible).  We see that $S' \cup P \cup Q \cup R \cup T_A \cup T'_A
   \cup T_B \cup T'_B$ is a cycle through $x, y, z$, a contradiction.

   {\noindent\bf Case 2:} we are not in Case 1.  We apply
   Lemma~\ref{l:menger2} in $G_A$ to $\{x_A, y_A\}$ and $\{s_A,
   z_A\}$.  Since we are not in Case~1, this gives two vertex-disjoint
   paths $T_A = x_A \dots z_A$ and $T'_A = y_A \dots s_A$.  We apply
   Lemma~\ref{l:menger2} in $G_B$ to $\{x_B, y_B\}$ and $\{s_B,
   z_B\}$.  Since we are not in Case~1, this gives two vertex-disjoint
   paths $T_B = x_B \dots s_B$ and $T'_B = y_B \dots z_B$.  We see
   that $S' \cup P \cup Q \cup R \cup T_A \cup T'_A \cup T_B \cup
   T'_B$ is a cycle through $x, y, z$, a contradiction.
\end{proof}

\section{Twins in 3-connected almost wheel-free graphs}
\label{sec:Twin}
%%%%%%%%%%%%%%%%%%%%%%%%%%%%

Our goal in this section is to prove Theorem~\ref{th:twin}.
Throughout all this section, $G$ is an almost wheel-free 3-connected
graph (recall that almost wheel-free graphs are defined before
Corollary~\ref{col:3con}).

\begin{lemma}
  \label{l:noTriangle}
  $G$ contains no triangle. 
\end{lemma}

\begin{proof}
  Let $u$, $v$ and $w$ be three pairwise adjacent vertices in $G$.
  Since $G$ is 3-connected, $v$ has a neighbor $v' $ distinct from $u$ and $w$.  In
  $G - v$, there is a 2-fan from $v'$ to $\{u, w\}$, that together
  with $v$ forms a wheel centered at $v$.  Similarly, there exist
  wheels centered at $u$ and $w$.  So, $|W(G)| \geq 3$, a
  contradiction.
\end{proof}

We denote by $K_{3, 3}\sm e$ the graph obtained from $K_{3, 3}$ by
removing one edge. 

\begin{lemma}
  \label{l:k33-e}
  If $G$ has a subgraph isomorphic to $K_{3, 3}\sm e$, then $G$ is
  isomorphic to $K_{3, 3}$. 
\end{lemma}

\begin{proof}
  Suppose that $G$ contains 6 vertices $a, b, c, x, y$ and $z$ such
  that there are all possible edges between $\{a, b, c\}$ and $\{x, y,
  z\}$ except possibly $ax$.  If there are no other vertices, then,
  since $G$ is 3-connected and there is no triangle
  by Lemma~\ref{l:noTriangle}, $a$ must be adjacent to $x$.  So, $G$ is
  isomorphic to $K_{3, 3}$. 

  \begin{claim}
    \label{c:Pxy}
    In $G - \{a, b, c\}$ there is no path from $x$ to $y$ and no path
    from $x$ to $z$; in $G - \{x, y, z\}$ there is no path from $a$
    to $b$ and no path from $a$ to $c$.
  \end{claim}
  
  \begin{proofclaim}
    If $P$ is a path of $G - \{a, b, c\}$ from $x$ to $y$, then $(xPy
    a z b x, c)$  and
    $(xPy a z c x, b)$ are wheels. 
    Hence $\{b, c\} \subseteq W(G)$, a contradiction because
    by Lemma~\ref{l:noTriangle}, $bc\notin E(G)$.  The other cases are
    symmetric.
  \end{proofclaim}

  If $G$ has more than 6 vertices, then without loss of generality,
  $x$ has a neighbor $v \notin \{ a, b, c\}$.  Let $P, Q$ be a 2-fan from $v$
  to $\{a, b, c, y, z\}$ in $G - x$.  If one of $P, Q$ is from $v$
  to $y$ or $z$, then, together with $x$, it forms a path that
  contradicts~(\ref{c:Pxy}).  So, $P, Q$ is in fact a 2-fan from $v$
  to $\{a, b, c\}$.  If one of $P$ and $Q$ ends in $a$, then $P
  \cup Q$ is a path that contradicts~(\ref{c:Pxy}).  So, $P, Q$ is in
  fact a 2-fan from $v$ to $\{b, c\}$.  Without loss of generality $P$ ends in $b$ and $Q$ in $c$.
  Then $(vPbyazcQv, x)$ is a wheel, so $x\in W(G)$. 
  Symmetrically, if $a$ has a neighbor $u \notin \{ x, y, z\}$, then $a\in W(G)$.
   Hence, either $a$ has a neighbor $u \notin\{x, y, z\}$, so $\{a, x\}
  \subseteq W(G)$ and $ax \in E(G)$, or $a$ has no neighbor $u \notin \{x, y,
  z\}$ and so $ax\in E(G)$ because $a$ has degree at least $3$.
  In both cases, $ax \in E(G)$.  Hence $(vPb x a z cQv, y)$ is a wheel.  So, $\{x, y\}\subseteq 
  W(G)$ which is a contradiction because $xy\notin E(G)$ by
  Lemma~\ref{l:noTriangle}.
\end{proof}

Two vertices $u$ and $v$ in a graph are \emph{twins} if they are
non-adjacent, of degree~3, and $N(u) = N(v)$.

\begin{lemma}
  \label{l:k23Twin}
  Suppose that two distinct vertices $a$ and $b$  of $G$ have three distinct common neighbors $x,
  y, z$.  Then $a$ and $b$ are twins.
\end{lemma}

\begin{proof}  
  By Lemma~\ref{l:noTriangle}, $x, y$ and $z$ are pairwise
  non-adjacent and $ab\notin E(G)$. 

  Suppose that $a$ has a neighbor $a'$ not in $\{x, y, z\}$.  In $G - a$,
  there is a 2-fan $P, Q$ from $a'$ to $\{x, y, z, b\}$.  We choose
  such a 2-fan which minimizes $|V(P) \cup V(Q)|$.  If
  the ends of $P$ and $Q$ are both in $\{x, y, z\}$, then $P$, $Q$,
  $a$ and $b$ form a wheel centered at $a$.  This is a contradiction
  because $a$ has degree at least~4.  Hence, we may assume up to
  symmetry that $P = a' \dots x$ and $Q = a' \dots b$.  

  \begin{claim}
    \label{c:zNoNeigh}
    $z$ has no neighbors in $\{y\} \cup V(P) \cup V(Q - b)$.
  \end{claim}

  \begin{proofclaim}
    We already showed that $z$ is not adjacent to $y$. 
     If $z$ has neighbors
    in $Q - b$, then there exists a 2-fan from $a'$ to $\{x, z\}$ in
    $G - a$, a contradiction as above.  If $z$ has a neighbor in $P$,
    then from the minimality of $P$ and $Q$, this neighbor must be the
    neighbor $x'$ of $x$ along $P$.  But then, $\{a, b, x', x, y, z\}$
    induces $K_{3, 3} \sm e$ or $K_{3, 3}$, and by
    Lemma~\ref{l:k33-e}, $G$ must be isomorphic to $K_{3, 3}$, a
    contradiction since $|V(G)|\geq 7$ because of $a'$.
  \end{proofclaim}

  \begin{claim}
    \label{c:uabv}
    If $u$ and $v$ are distinct vertices of $G[V(P) \cup V(Q) \cup
    \{a, y\}]$, then there exists a path $R$ from $u$ to $v$ in
    $G[V(P) \cup V(Q) \cup \{a, y\}]$, that contains $a$ and~$b$.
  \end{claim}

  \begin{proofclaim}
    The outcome is clear when $u, v, a, b$ or $v, u, a, b$ appear in
    this order in some cycle of $G[V(P) \cup V(Q) \cup \{a, y\}]$.

    Suppose $u \in V(Q)$.  If $v\in V(Q) \cup \{a\}$, then $u, v, a, b$
    or $v, u, a, b$ appear in this order in $a' Q b y a a'$. So
    suppose $v \notin V(Q) \cup \{a\}$, and in particular $v\notin \{a',
    b\}$.  If $v\in V(P)$, then $R = u Q b y a x P v$.  If $v = y$,
    then $R = u Q b x a v$.  So, we may assume $u\notin V(Q)$ and
    symmetrically $v\notin V(Q)$.

    Suppose $u \in V(P)$.  Recall $u\neq a'$.  If $v\in V(P) \cup
    \{a\}$, then $u, v, a, b$ or $v, u, a, b$ appear in this order in
    $a' P x b y a a'$. If not, then $v = y$, and $R = u P x b Q a' a
    v$.  So we may assume $u\notin V(P)$, and symmetrically $v\notin
    V(P)$.

    Then $\{u, v\} = \{a, y\}$ and  $R=axby$.
  \end{proofclaim}

  Since $G$ is 3-connected and by~(\ref{c:zNoNeigh}), $z$ has a
  neighbor $z' \notin V(P) \cup V(Q) \cup \{a, y\}$.  Let $S= z' \dots
  u, T=z'\dots v$ be a 2-fan in $G - z$ from $z'$ to $V(P) \cup V(Q)
  \cup \{a, y\}$.  Let $R$ be the path obtained in~(\ref{c:uabv}).
  Now, $S\cup T\cup R$ is a cycle in which $z$ has at least three
  neighbors (namely $z'$, $a$ and $b$).  Hence $z$ is
  the center of some wheel of $G$.  Similarly, the existence of a
  wheel centered at $y$ can be proved.  So $\{y, z\} \subseteq W(G)$, a
  contradiction since $yz \notin E(G)$.
\end{proof}

\begin{lemma}
  \label{l:3sep}
  Let $F$ be a fragment of $G$ with $N(F) = \{a, b, c\}$.  Consider
  the graph $G_F$ built from $G[F \cup \{a, b, c\}]\setminus
  \{ab,bc,ca\}$ as follows.  If $a$ (resp.\ $b$, $c$) has at least two
  neighbors in $F$, then add a new vertex $a'$ (resp.\ $b'$, $c'$)
  adjacent to $a$ (resp.\ $b$, $c$), otherwise put $a' =a$ (resp.\
  $b'=b$, $c'=c$).  Add two new vertices $d, d'$ and link them to
  $a'$, $b'$ and $c'$.

  Then $G_F$ is an almost wheel-free 3-connected graph.
\end{lemma}

\begin{proof}
  Note that since $G$ is 3-connected, every vertex of $\{a, b, c\}$
  has at least one neighbor in $F$.  Note that in all cases, $a'$, $b'$ and $c'$ have
  degree~3 and are pairwise non-adjacent.

  Let us first prove that $G_F$ is 3-connected.  
  Suppose for a contradiction that $G_F$ has a $2$-cutset $\{w, w'\}$.
  Observe that if
  $x, y\in F$, then there exist three internally vertex-disjoint paths
  from $x$ to $y$ in $G$, and at most one of them has vertices in
  $V(G) \sm (F \cup \{a, b, c\})$.  This path can be rerouted through
  $d$ to obtain a path of $G_F$.  It follows that in $G_F$, any pair
  of vertices from $F$ can be linked by three vertex-disjoint paths.
  Hence all the vertices
  from $F\sm \{w, w'\}$ are in the same component of $G_F - \{w,
  w'\}$.  Therefore, to get a contradiction,  it is
  sufficient to show that any vertex of $\{a,b,c,a',b',c',d,d'\}\setminus \{w,w'\}$ can be linked
  by a path of $G_F - \{w, w'\}$ to some vertex of $F\sm \{w, w'\}$.
    If $\{w, w'\}
  \subseteq F$, then at least one of
  $a, b, c$ has a neighbor in $F\sm \{w, w'\}$, because $G$ is $3$-connected.
   So $G_F - \{w, w'\}$ is connected.  If $w\in F$ and $w'
  \notin F$, then $G_F - \{w, w'\}$ is connected, unless $w$ is the
  unique neighbor of $a$ in $F$, $a'\neq a$ and $w=a'$ (up to a
  symmetry).  But this contradicts the way $G_F$ is constructed,
  because when $a$ has a unique neighbor in $F$, then $a = a'$.
  When $w, w'\notin F$, one can easily see again that $G_F - \{w, w'\}$ is
  connected.  This proves that $G_F$ is 3-connected.

  Let us now prove that $W(G_F) \subseteq W(G)$.  Let $w\in W(G_F)$ and
  let $C$ be the rim of a wheel centered at $w$.  If $w = d$, then
  $C$ must go through $a'$, $b'$ and $c'$.  Since these vertices have
  degree~3, are pairwise non-adjacent and are adjacent to $d'$, there
  is a contradiction because the cycle $C$ must contain three edges
  incident to $d'$.  So $w\neq d$, and similarly $w\neq d'$.  If
  $w=a'$, then $C$ must go through the edges $db', dc', d'b', d'c'$, a
  contradiction.  So $w\neq a'$, and similarly $w\neq b'$ and $w\neq
  c'$.  If $w=a$, then we know $a\neq a'$.  If $C$ is contained in
  $G[F]$, then $a\in W(G)$.  If not, then $C\cap G[F]$ is a path $P$
  from $b$ to $c$ containing at least two neighbors of $a$.  Let $x\in
  V(G)\sm (F \cup \{a, b, c\})$ be a neighbor of $a$.  In $G - a$,
  consider a 2-fan $Q, R$ from $x$ to $\{b, c\}$.  Then $P\cup Q\cup
  R$ is a cycle of $G$ in which $a$ has at least three neighbors.
  Hence $a\in W(G)$.  If $w\in F$, then a wheel of $G$ centered at $w$
  can be obtained by replacing some path of $C$ with both ends in
  $\{a, b, c\}$ with a path from $G - F $ with the same ends.  Hence,
  $w\in W(G)$.

  We proved that $W(G_F) \subseteq W(G)$.  But every vertex of $F$ has
  the same degree in $G$ and $G_F$, and the vertices of $\{a,b,c\}$
  have degree in $G_F$ no larger than in $G$.  Therefore, $W(G_F)$ is
  either empty, or made of a single vertex of degree~3, or made of two
  adjacent vertices, both of degree~3. In other words, $G_F$ is almost
  wheel-free.
\end{proof}

Note that Lemma~\ref{l:3sep} is not so easy to use in a proof by
induction, because the graph $G_F$ may have more vertices and edges
than~$G$.  Also $G=G_F$ is possible.  A vertex is \emph{close to a
  twin} if it is either a member of a pair of twins or adjacent to a
member of a pair of twins.

\begin{lemma}
  \label{l:denseT}
  If every vertex of degree~3 in $G$ that is not in $W(G)$ is
  close to a twin, then $G$ contains two disjoint pairs of twins.
\end{lemma}

\begin{proof}
  Since $G$ is 3-connected, $|V(G)| \geq 4$.  If $|V(G)| = 4$, then
  $G$ is isomorphic to $K_4$, and $G$ contains a triangle, a
  contradiction to Lemma~\ref{l:noTriangle}.  So, $|V(G)| \geq 5$.
  Since $G$ is minimally 3-connected, by Theorem~\ref{mader}, it
  follows that $G$ contains at least $\lceil 12/5 \rceil = 3$ vertices of
  degree~3.  One of them, say $v$, is not in $W(G)$.  Hence, $v$ is close to
  a twin.  It follows that $G$ contains a pair of twins $\{a, b\}$.
  Let $x, y, z$ be the three neighbors of $a$ and $b$. 

  Suppose that $x$ and $y$ have a common neighbor $c$ distinct from $a$ and $b$.   
  Then $G$ contains a subgraph isomorphic to $K_{3, 3} \sm
  e$.  Hence, by Lemma~\ref{l:k33-e}, $G$ is isomorphic to $K_{3, 3}$,
  so it contains two disjoint pairs of twins.  Therefore, we may
  assume that $x, y$ have no common neighbors (except $a$ and $b$),
  and similarly, $x, z$ and $y, z$.  In particular, no pair of twins
  of $G$ contains $x$, $y$, or~$z$.
  Let $R = V(G) \sm \{a, b, x, y, z\}$.  Note that $|R| \geq 3$
  because of the neighbors of $x$, $y$ and $z$.  
  
  We claim that $R$
  contains a vertex $u\notin W(G)$ of degree~3 in $G$.  Suppose first
  that $G[R\sm W(G)]$ has no vertex of degree at most~1.  Then,
  $G[R\sm W(G)]$ contains a cycle $C$, which is also a cycle in $G$.
  By Corollary~\ref{col:3con} $G$ is minimally 3-connected, and so, according to
  Theorem~\ref{maderCycle}, cycle $C$ contains a vertex $u$
  whose degree (in $G$) is~3.  Hence, we are left with the case when
  $G[R\sm W(G)]$ has a vertex $u$ of degree at most 1.  The degree of
  $u$ in $G$ is at most~3. Indeed, $u$ is adjacent to at most one
  vertex among $x, y, z$ and to at most one vertex in $W(G)$ because
  $W(G)$ is a clique and by Lemma~\ref{l:noTriangle}, there is no
  triangle in $G$.  This proves our claim.
  
  Now $u$ is not in $W(G)$, is close to a twin, so it must be a member of
  a pair of twins of $G$, or adjacent to some member of a pair of
  twins of $G$.  This pair of twins is in $R$, so, it is disjoint from
  $\{a, b\}$. 
\end{proof}

\begin{theorem}
  \label{th:twin}
  If $G$ is an almost wheel-free 3-connected graph, then $G$ contains
  two disjoint pairs of twins.
\end{theorem}

\begin{proof}
  We consider a minimum counter-example $G$, and we look for a
  contradiction. 
  
  \begin{claim}
    \label{i:f3}
    No fragment $F$ of $G$ is such that $|F| \geq 6$, $|\overline{F}|
    \geq 2$ and $F$ contains a pair of twins of $G$.
  \end{claim}

  \begin{proofclaim}
    For suppose that such a fragment $F$ exists with $N(F) = \{a, b, c\}$.
    So, $\overline{F}$ is also a fragment of $G$ and $N(\overline{F})
    = \{a, b, c\}$.  Consider the graph $G_{\overline{F}}$ built as in
    Lemma~\ref{l:3sep}.  Since $|F|\geq 6$, we
    have $|V(G_{\overline{F}})| < |V(G)|$.  By Lemma~\ref{l:3sep},
    $G_{\overline{F}}$ is almost wheel-free and 3-connected.
    
    By the minimality of $G$, $G_{\overline{F}}$ contains two disjoint
    pairs of twins. None of them is a pair of twins in $G$, for
    otherwise with the one in $F$, $G$ would have two pairs of twins,
    a contradiction.  Hence one pair of twins is $\{d, d'\}$ and the
    second one is $\{a, b\}$, $\{a, c\}$ or $\{c, b\}$, because if
    $a'\neq a$ (resp.\ $b'\neq b$, $c'\neq c$), then $a'$ (resp.\
    $b'$, $c'$) is in no pair of twins.  Without loss of generality
    $\{a, b\}$ is a pair of twins of $G_{\overline{F}}$.  This means
    that $a$ and $b$ are adjacent to $d$, to $d'$ and to a vertex $d''
    \in \overline{F}$, and that $d''$ is the unique neighbor of $a$
    and $b$ in $\overline{F}$.  It follows that $\{d'', c\}$ is a
    cutset of $G_{\overline{F}}$, a contradiction, unless
    $\overline{F} = \{d''\}$, which is also a contradiction because we
    suppose $|\overline{F}| \geq 2$.
  \end{proofclaim}

  By Lemma~\ref{l:denseT}, and because $G$ does not contain two
  disjoint pairs of twins, there exists a vertex $v$ in $V(G)\sm W(G)$ of
  degree~3 and not close to a twin.  Let $x$, $y$ and $z$ be the three
  neighbors of $v$.  Note that $G - v$ is 2-connected and in $G -
  v$, no cycle goes through $x, y, z$ (because such a cycle would be
  the rim of a wheel centered at $v \notin W(G)$).  Hence, by
  Theorem~\ref{th:3inaCycle}, there is a splitter $A = \{x_A, y_A,
  z_A\}$, $B = \{x_B, y_B, z_B\}$ for $x, y, z$ in $G - v$.  We
  denote by $X, Y, Z$ the components of $G - (A \cup B)$ that contain
  $x, y, z$ respectively.

  \begin{claim}
    \label{c:1or4}
    Either $|Y| = 1$, or $|Y| \geq 4$, or $Y= \{y, y', y''\}$ and
    $y_Ay'$, $y_Ay''$, $y_By'$, $y_By'' \in E(G)$. 
  \end{claim}

  \begin{proofclaim}
    Suppose $Y$ has cardinality~2, say $Y = \{y, y'\}$.  Since $y'$
    has degree at least~3 and is non-adjacent to $v$, $y'$ must be
    adjacent to $y_A$, $y_B$, and $y$.  Since $y$ also has degree at
    least~3, it must be adjacent to at least one of $y_A, y_B$.
    Hence, $G$ contains a triangle, a contradiction to
    Lemma~\ref{l:noTriangle}.

    Suppose $Y$ has cardinality~3, say $Y = \{y, y', y''\}$.  If $yy'
    \notin E(G)$ then, since $y'$ has degree at least~3, it must be
    adjacent to $y''$, $y_A$ and $y_B$.  Also, $y''$ must be adjacent
    to at least one of $y_A$ or $y_B$, so $G$ contains a triangle, a
    contradiction to Lemma~\ref{l:noTriangle}.  Hence, $yy'\in E(G)$
    and similarly, $yy'' \in E(G)$.  So by Lemma~\ref{l:noTriangle}
    $y'y''\notin E(G)$. Since $y'$ and $y''$ have degree at least~3,
    $y_Ay', y_Ay'', y_By', y_By'' \in E(G)$ as claimed.
  \end{proofclaim}

  \begin{claim}
    \label{c:Xok}
    Either $|X| = 1$ or $X$ contains a pair of twins of~$G$.
  \end{claim}

  \begin{proofclaim}
    Suppose $|X|>1$.  Note that $X$ is a fragment of $G$.  So let us
    build the graph $G_X$ as in Lemma~\ref{l:3sep} (with more
    convenient names given to the vertices): start with $G[X \cup \{v,
    x_A, x_B\}]$ and add two vertices $y'$ and $z'$ and the edges
    $vy'$ and $vz'$.  If $x_A$ has at least two neighbors in $X$, then
    add a new vertex $x'_A$ and the edges $x_Ax'_A$, $x'_Ay'$ and
    $x'_Az'$; otherwise, set $x'_A=x_A$ and add the edges $x_Ay'$ and
    $x_Az'$.  If $x_B$ has at least two neighbors in $X$, then add a
    new vertex $x'_B$ and the edges $x_Bx'_B$, $x'_By'$ and $x'_Bz'$;
    otherwise, set $x'_B=x_B$ and add the edges $x_By'$ and $x_Bz'$.
    By Lemma~\ref{l:3sep}, $G_X$ is almost wheel-free and 3-connected.
 
    We claim that $|V(G_X)| < |V(G)|$.  Observe that $G_X$ has at most
    four vertices not in $G$, the ones of $\{z', y', x'_A, x'_B\}$.
    Suppose for a contradiction that $|V(G_X)| \geq |V(G)|$.  Then
    $|Y|+|Z|\leq 4$. By (\ref{c:1or4}), one of $Y$ and $Z$, say $Z$,
    has cardinality $1$.  If $|Y|>1$, then by (\ref{c:1or4}), $Y= \{y,
    y', y''\}$ and $y_Ay'$, $y_Ay''$, $y_By'$, $y_By'' \in E(G)$.
    Moreover $A= \{x_A\}$ and $B= \{x_B\}$ for otherwise $|V(G_X)| <
    |V(G)|$. Thus $x_A$ and $x_B$ have three common neighbors in $G$,
    namely $y', y''$ and~$z$, and also have degree at least~4. This
    contradicts Lemma~\ref{l:k23Twin}.  Hence $|Y| = |Z| = 1$. By
    Lemma~\ref{l:noTriangle}, $y_Ay_B\notin E(G)$ and $z_Az_B\notin
    E(G)$.  If $y_A=z_A$ and $y_B=z_B$, then $y$ and $z$ have three
    common neighbors in $G$, namely $x_A$, $x_B$ and $v$.  Hence, by
    Lemma~\ref{l:k23Twin}, they form a pair of twins of $G$, so $v$ is
    close to a twin, a contradiction to the choice of $v$.  Hence by
    symmetry, we may assume that $|A| = 3$.  Hence $|V(G_X)| <
    |V(G)|$, unless $|B|=1$, $x_A \neq x'_A$, $x_B \neq x'_B$, and
    $V(G) = A \cup B \cup X \cup \{y, z, v\}$.  Since there is no
    triangle, one of $y_A, z_A$ is of degree~2, a contradiction.  This
    proves the claim.

    Now, $G_X$ is almost wheel-free, 3-connected and smaller than $G$.
    Hence, by the minimality of $G$, $G_X$ contains two disjoint pairs
    of twins.  One of them is $\{y', z'\}$.  The other one is either
    in $X$, in which case it is also a pair of twins of $G$ (what we
    want to prove) or is $\{x_A, x_B\}$.  In the later case, $x_A$ and
    $x_B$ have degree~3 in $G_X$, so they are both adjacent to a
    unique same vertex $x'$ in $X$.  Then $\{x', v\}$ is a cutset of
    $G_X$, that separates $X\sm \{x'\}$ from $x_A$, a contradiction to
    the $3$-connectivity of $G_X$.
  \end{proofclaim}

  \begin{claim}
    \label{c:XYZ1}
   ~$|X| = |Y| = |Z| = 1$. 
  \end{claim}

  \begin{proofclaim}
    By~(\ref{c:Xok}) every set among $X$, $Y$ or $Z$ of cardinality at least two contains a pair of twins. Hence we may assume
    that $|Y| = |Z| = 1$. Thus $y_Ay_B$ and $z_Az_B$ are not edges by
    Lemma~\ref{l:noTriangle}. Suppose $|X|> 1$.  Note that
    by~(\ref{c:Xok}), $X$ contains a pair of twins of $G$.

    Suppose first that $|A| = |B| = 3$.  Hence, since $G$ is
    3-connected, every component $D$ of $G - (A \cup B \cup \{v\})$
    satisfies $N(D) \subseteq A$ or $N(D) \subseteq B$.  Let $C_A$ be
    the union of all components $D$ such that $N(D) \subseteq A$.
    Because $G$ has minimum degree~$3$ and no triangle, there must be
    at least two vertices in $C_A$.  Hence, $F= V(G) \sm (C_A \cup
    \{x_A, y_A, z_A\})$ contradicts~(\ref{i:f3}).

    Suppose that $|B| = 1$ and $|A| = 3$.  The vertices $x_A$, $y_A$
    and $z_A$ have degree at least~3, and by~(\ref{i:f3}), there is at
    most one vertex in the union of all components $D$ such that $N(D)
    \subseteq A$ in $G - (A \cup B \cup \{v\})$.  If such a component
    $D$ exists, then there is a common neighbour to $x_A$, $y_A$ and
    $z_A$ and those three vertices are pairwise non-adjacent.  Hence
    there must be a component $D'$ such that $x_B \in N(D')$ because
    of their degrees.  If no component $D$ such that $N(D) \subseteq
    A$ exists, then since $x_A$, $y_A$ and $z_A$ do not form a
    triangle, there also must be a component $D'$ such that $x_B \in
    N(D')$.  Let $x'_B$ be a neighbor of $x_B$ in $D'$.  In $G - x_B$,
    consider a 2-fan from $x'_B$ to $\{x_A, y_A, z_A\}$.  If this
    2-fan is formed of $P=x'_B \dots y_A$ and $Q=x'_B \dots z_A$, then
    $(vyy_APx'_BQz_Axv, x_B)$ is a wheel, a contradiction because
    $x_B$ has degree at least $4$.  Hence, we may assume that the
    union of the two paths of this 2-fan is a path $P$ from $x_A$ to
    $y_A$ with internal vertices in $C_B$ and going through~$x'_B$.
    Now, consider a neighbor $x'\in X$ of $x_B$, and a 2-fan in $G -
    \{x_B\}$ from $x'$ to $\{v, x_A\}$.  The union of the two paths of
    this 2-fan is a path $Q$ from $v$ to $x_A$, with interior in $X$,
    and that goes through $x'$.  Then $v Q x_A P y_A y v$ is a cycle
    that contains three neighbors of $x_B$ (namely $y$, $x'$ and
    $x'_B$), a contradiction since $x_B$ has degree at least~4.

    Similarly, we get a contradiction if $|B| = 1$ and $|A| = 3$. So
    $|A| = |B| = 1$ and $\{y, z\}$ is a pair of twins, a contradiction to
    the choice of $v$.
 \end{proofclaim}

\begin{claim}
  \label{c:AB3}
  ~$|A| = |B| = 3$.
\end{claim}

\begin{proofclaim}
  By the choice of $v$,  $\{x, y\}$ is not a pair of twins of $G$,  and so it is impossible to have $|A| = 1$ and $|B| = 1$.  
   If $x_A = y_A = z_A$, then by Lemma~\ref{l:k23Twin}, $\{v,
  x_A\}$ is a pair of twins of $G$, a contradiction to the choice of $v$.
   Similarly, $x_B = y_B = z_B$ is impossible.
  Hence, $|A| = |B| = 3$.
\end{proofclaim}

We are now ready to finish the proof.  We know by~(\ref{c:XYZ1}) that
$|X| = |Y| = |Z| = 1$ and by~(\ref{c:AB3}) that $|A| = |B| = 3$.
Hence, every component $D$ of $G - (A \cup B \cup \{v, x, y, z\})$
satisfies $N(D) \subseteq A$ or $N(D) \subseteq B$, and there are no
edges between $A$ and $B$.  Let $C_A$ (resp.\ $C_B$) be the union of
all components $D$ such that $N(D) \subseteq A$ (resp.~$N(D) \subseteq
B$).  Then, $F = C_B \cup \{x_B, y_B, z_B, v\}$ is a fragment of $G$,
and we build the graph $G_F$ as in Lemma~\ref{l:3sep}.  Note that
$N(F) = \{x, y, z\}$ and each of $x$, $y$ and $z$ has exactly one
neighbor in $\overline{F}$.  So, the graph $G_F$ is obtained by adding
to $G\sm F$ two vertices $d$ and $d'$, and by linking them to $x$, $y$
and $z$. We obtain a graph smaller than $G$, and by the minimality of
$G$, it must contain two disjoint pairs of twins.  One of them is
$\{d, d'\}$, the other one must be in $C_A$ and is in fact a pair of
twins of $G$.  Hence, $C_A$ contains a pair of twins of $G$, and by a
symmetric argument, so does $C_B$.  Hence $G$ contains two disjoint
pairs of twins, a contradiction.
\end{proof}

\section{Proof  of Theorem~\ref{th:2orT}}
\label{sec:col}
%%%%%%%%%%%%%%%%%%%%

To prove Theorem~\ref{th:2orT}, we actually prove the following stronger theorem.
 \begin{theorem}\label{thm:stronger}
 Every wheel-free graph on at least two
  vertices contains either

  \begin{enumerate}[(i)]
  \item\label{o:22} two vertices of degree at most 2; or
  \item\label{o:121T} one vertex of degree at most 2 and one pair of
    twins; or
  \item\label{o:2T} two disjoint pairs of twins.
  \end{enumerate}
\end{theorem}

\begin{proof}
 By induction on the number of vertices.

  If $|V(G)|=2$, then~(\ref{o:22}) obviously holds.  If $G$ is
  not 2-connected, then the conclusion follows easily from the
  induction hypothesis applied to connected components or blocks
  of~$G$.  If $G$ is 3-connected, then the conclusion holds by
  Theorem~\ref{th:twin}.  Hence, we may assume that the connectivity
  of $G$ is~2.  It is enough to prove that every end of $G$ contains
  either a vertex of degree~2 or a pair of twins of $G$.  So, let $F$
  be an end of $G$ and $N(F) = \{a, b\}$.  If $|F| = 1$, then $F$
  contains a vertex of degree~2, so suppose $|F| \geq 2$.  Build $G_F$
  as in Lemma~\ref{l:2conBlock}.

  \noindent{\bf Case 1:} $ab\in E(G)$.  Then, $G_F$ is a subgraph of
  $G$, so $G_F$ is wheel-free and 3-connected.  By
  Theorem~\ref{th:twin}, $G_F$ contains 2 disjoint pairs of twins.  If
  one of them is disjoint from $\{a, b\}$, then it is a pair of twins
  of $G$.  Hence, we may assume the two disjoint pairs of twins in $G_F$ are
  $\{a, a'\}$ and $\{b, b'\}$ for some $a'\neq b'$.  Note that $ab', ba'
  \in E(G)$, so $a'b' \in E(G)$.  Let $a''$ be the third common
  neighbor of $b, b'$ and $b''$ be the third common neighbor of $a,
  a'$.  Observe that a subgraph of $G_F$ on $\{a, a', a'', b, b',
  b''\}$ is isomorphic to $K_{3, 3} \sm e$.  Hence, by
  Lemma~\ref{l:k33-e}, $G_F$ is in fact isomorphic to $K_{3, 3}$.  It
  follows that $\{a', a''\}\subseteq F$ is a pair of twins of $G$.

\vspace{1ex}

\noindent{\bf Case 2:} $ab\notin E(G)$.  Then  $G[F
\cup \{a, b\}]=G_F\setminus ab$.  We
first prove the following.

  \begin{claim} 
    \label{c:noPaira}
    No pair of twins of $G_F \sm ab$ contains $a$ (resp.\ $b$).
  \end{claim}

  \begin{proofclaim}
    Otherwise, let $\{a, a'\}$ be a pair of twins in $G_F \sm ab$.
    Let be the three neighbours of $a$ in $F$.  Since $G_F$ is
    3-connected, $x$ has degree 3 in $G_F$ and so has a neighbour $x'$
    in $F\setminus \{a,b\}$.
    
     Suppose $a'=b$. A 2-fan in $G_F - x$, from
    $x'$ to $\{y, z\}$ together with $a$, $b$ and $x$ form a wheel of
    $G_F$ centered at $x$, a contradiction to Lemma~\ref{l:2conBlock}.
    Hence $a'\neq b$.

    In $G_F - x$, let $P,
    Q$ be a 2-fan from $x'$ to $\{y, z, b\}$.  Up to symmetry, $P= x'
    \dots y$.  If $Q= x' \dots b$, then $(x'Pya'zabQx',x)$ is a wheel in $G_F$, a contradiction to  Lemma~\ref{l:2conBlock}.
    Hence $Q= x' \dots z$.  Let  $P', Q'$ be a
    2-fan in $G_F - a$ from $b$ to $\{x, y, z\}$.  Note that $(V(P)
    \cup V(Q)) \sm \{y, z\}$ is disjoint from $V(P') \cup V(Q')$, for
    otherwise, a 2-fan from $x'$ to $\{y, b\}$ or to $\{z, b\}$
    exists, giving a contradiction as above.  Hence we may assume $P'
    = b \dots y$ and $Q' = b \dots x$ or $Q' = b \dots z$.  
    Then $(bP'y'Px'Qzaxb, a')$ or $(bP'yPx'xazb, a')$ respectively is a wheel in $G$, a contradiction.
    
  \end{proofclaim}  

  If $a$ or $b$ has at least three neighbors in $F$, then we apply the
  induction hypothesis to $G_F\sm ab$.  If $G_F\sm ab$
  contains two vertices of degree at most~2, one of them is in $F$.  Otherwise, $F$ contains a pair of twins in $G_F\sm ab$. 
  By~(\ref{c:noPaira}) this pair does not intersect $\{a, b\}$ and thus is also a pair of twins in $G$.

  Hence we may assume that $a$ and $b$ have both exactly two
  neighbors in~$F$.  Thus $G_F$ is almost wheel-free because $W(G_F)\subseteq \{a,b\}$ by Lemma~\ref{l:2conBlock}. 
   By Theorem~\ref{th:twin}, $G_F$ contains two disjoint pairs of twins.
  If one of them is in $F$, it is also a pair of twins in $G$.
  If not, then the two disjoint pairs of twins in $G_F$ are
  $\{a, a'\}$ and $\{b, b'\}$ for some $a'\neq b'$.  As in Case~1, one shows that $G_F$ has a subgraph
  isomorphic to $K_{3, 3}\sm e$, and so by Lemma~\ref{l:k33-e}, $G_F$ is
  isomorphic to $K_{3, 3}$.  It follows that $F$ contains a pair of
  twins of $G$.
\end{proof}

Remark  that Theorem~\ref{thm:stronger} is best possible as shown by the three graphs represented on
  Figure~\ref{fig:3options}.

 \begin{figure}[hbtp]
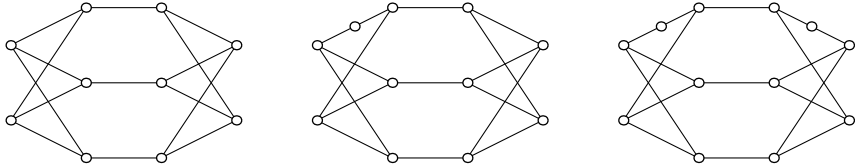

\begin{center}
\includegraphics{figWheels.3}\hspace{2em}
\includegraphics{figWheels.11}\hspace{2em}
\includegraphics{figWheels.12}\vspace{3ex}
\end{center}
\caption{Three wheel-free graphs.\label{fig:3options}}
\end{figure}

\section{When $K_4$ does not count as a wheel}
\label{sec:long}
%%%%%%%%%%%%%%%%%%%

A \emph{long wheel} is a wheel whose rim is a cycle of length at
least~4.  Observe that $K_4$ is the only wheel that is not a long
wheel.  Note that in several articles, the word \emph{wheel} is used
for what we call here \emph{long wheel}.

\begin{lemma}
  \label{l:block}
  Let $G$ be a long-wheel-free graph.  Then every block of $G$ is
  wheel-free or isomorphic to $K_4$.  
\end{lemma}

\begin{proof}
  Let $H$ be a block of $G$.  If $H$ does not contain $K_4$, then it
  is obviously wheel-free.  So, suppose that $H$ contains a subgraph
  $H'$ isomorphic to $K_4$.  If $H\neq H'$, then let $v\in V(H) \sm
  V(H')$.  A long wheel in $G$ is obtained by taking the union of $H'$
  and a 2-fan from $v$ to $H'$, a contradiction. 
\end{proof}

\begin{theorem}
  If a graph does not contain a long wheel as a subgraph, then it is
  4-colorable.  
\end{theorem}

\begin{proof}
  It is enough to prove that every block $H$ of $G$ is 4-colorable.
  This is obvious if $H$ is isomorphic to $K_4$.  Otherwise, by
  Lemma~\ref{l:block}, $H$ is wheel-free, so it is even 3-colorable by
  Theorem~\ref{th:col}.
\end{proof}

Note that to prove the theorem above, we do not need the full strength
of Theorem~\ref{th:col}.  Knowing that wheel-free graphs are
4-colorable is enough, and follows easily from
Theorem~\ref{th:Turner} or Theorem~\ref{th:deg3}.

\section*{Acknowledgement}

Thanks to Louis Esperet and Mat\v ej Stehl\'ik for useful discussions and
for pointing out to us a graph with chromatic number~4, no triangle and no
wheel as an induced subgraph.  Thanks to Fr\'ed\'eric Maffray for
pointing out to us the paper of Turner~\cite{turner:05}.

\end{document}